\documentclass[12pt]{amsart}
\usepackage[T1]{fontenc}
\usepackage[utf8]{inputenc}
\usepackage[
  backend=biber,
  style=numeric,
  sorting=none,
  giveninits=false,
  doi=true,
  isbn=true,
  url=true,
  eprint=true
]{biblatex}
\usepackage{amsmath, amssymb, amsthm}
\usepackage{mathrsfs}
\usepackage{bbm}
\usepackage{hyperref}
\usepackage{geometry}
\usepackage{enumitem}
\usepackage{booktabs}
\usepackage{array}
\usepackage{tikz-cd}   

\newtheorem{theorem}{Theorem}[section]
\newtheorem{proposition}[theorem]{Proposition}
\newtheorem{lemma}[theorem]{Lemma}

\newtheorem{definition}[theorem]{Definition}

\theoremstyle{remark}
\newtheorem{remark}[theorem]{Remark}

\newcommand{\Perv}{\mathrm{Perv}}
\newcommand{\rat}{\mathrm{rat}}

\newcommand{\var}{\mathrm{var}}

\newcommand{\Ext}{\mathrm{Ext}}

\newcommand{\Cone}{\mathrm{Cone}}

\newcommand{\Id}{\mathrm{Id}}

\begin{document}

\title[Interacting Multi-Node Conifold Light Sectors]{Interacting Multi-Node Conifold Light Sectors}

\author{Abdul Rahman}
\thanks{Email: arahman@alum.howard.edu}

\begin{abstract}
We study finite-node conifold degenerations of Calabi--Yau threefolds from the point of view of interacting light sectors. Although each ordinary double point contributes a rank-one local vanishing sector, the corrected global object need not assemble as a freely independent sum of nodewise pieces. Using the corrected perverse and mixed-Hodge-module degeneration package, the global gluing law for corrected extension classes, and the rigid-flexible atom decomposition on the \(F\)-bundle side, we define an interacting multi-node light-sector package and prove a block-reduced structure theorem. In the block-separated cycle family, the finite-node package separates into two logically distinct layers: relation collapse, controlled by a common relation lattice on the corrected-extension, smoothing, and resolution sides, and residual interaction among the surviving global sectors, controlled by a reduced block interaction matrix on the transport and atom sides. The result isolates the geometric and Hodge-theoretic precursor of coupled conifold light states and provides the mathematical input for a later multi-node reformulation of Strominger's conifold mechanism.
\end{abstract}
\maketitle
\tableofcontents

\section{Introduction}

In Strominger's analysis, the singularity of the low-energy effective theory is resolved by integrating in the light state that becomes massless at the conifold, so the one-node picture is governed by a single local vanishing sector and its associated monodromy \cite{Strominger95}. On the sheaf-theoretic side, the corrected degeneration object is constructed from nearby and vanishing cycles and isolates the singular contribution as a corrected extension of the invariant bulk sector \cite{RahmanSchoberPaper,RahmanPerverseNearbyCycles}. On the mixed-Hodge-module side, the same corrected object admits a canonical lift, so the local singular sector is incorporated into a limiting Hodge-theoretic package rather than appearing only at the level of perverse sheaves \cite{RahmanMixedHodgeModules}. On the \(F\)-bundle side, the same local mechanism is reflected by Picard--Lefschetz monodromy, Stokes data, and the rigid-flexible atom decomposition \cite{RahmanHodgeAtoms2026}. Thus the one-node picture is not the missing part of the story: its local singular sector is already visible across several compatible mathematical realizations.

What Strominger explicitly left open is the harder finite-node regime. After writing down the standard simple-conifold monodromy, he observed that there are ``more complicated types of conifolds with multiple degenerations and different monodromies'' and that further analysis is required to determine whether they can be resolved in a similar fashion \cite[p.~4, footnote 2]{Strominger95}. That is exactly where the present paper enters. In the multi-node case, if
\[
\Sigma=\{p_1,\dots,p_r\},
\]
then each node contributes a formal rank-one local singular sector, so at first sight one obtains a free nodewise collection of local sectors. But this is only the ambient local picture. The corrected global object comes from one projective degeneration, not from \(r\) unrelated local models, and therefore the nodewise sectors need not remain freely independent after global assembly. The central mathematical problem is to determine the actual finite-node light-sector space realized by geometry, together with the interaction structure governing its gluing, transport, and splitting behavior.

Earlier works provide the ingredients for that analysis, but not yet the unified finite-node object itself. On the corrected-extension side, the global-gluing framework shows that the corrected extension class may be forced into a smaller relation-controlled subspace determined by cycle-node incidence and global compatibility data \cite{RahmanGlobalGluing}. On the mixed-Hodge and \(F\)-bundle side, the Hodge-atoms framework shows that the same corrected degeneration object need not split into independent flexible pieces; instead, the corresponding Stokes and Picard--Lefschetz data may be noncommuting and the flexible sectors may mix nontrivially \cite{RahmanHodgeAtoms2026}. These results strongly indicate that the finite-node case carries a genuine interaction structure invisible in the naive free nodewise picture, but prior to the present paper that structure had not yet been isolated as a single mathematical package.

The purpose of this paper is to identify that package. We define and study the finite-node \emph{interacting light-sector package}: the package that records, in one finite-node object, the local rank-one node sectors, their global assembly law, their transport-side interaction data, and their atom-side mixing behavior. This is the new contribution of the paper: we isolate the finite-node package compatibly realized by corrected global gluing, by Picard--Lefschetz/Stokes transport, and by atom-side splitting behavior, and we organize these realizations into a common comparison framework. More sharply, in the block-separated cycle family the paper proves that the finite-node package decomposes into two logically distinct layers:
\[
\text{relation collapse}
\qquad+\qquad
\text{residual interaction among surviving sectors}.
\]
The first layer is controlled by the common relation lattice on the corrected-extension, smoothing, and resolution sides, while the second is controlled by a reduced block interaction matrix on the transport and atom sides. In this way, we supply the mathematical foundation needed for a later reformulation of Strominger’s multi-node conifold mechanism.

\subsection{From one light sector to many}

For a single ODP, the corrected singular quotient is rank one on the perverse side \cite{RahmanSchoberPaper,RahmanPerverseNearbyCycles,RahmanMixedHodgeModules}, lifts to a rank-one Tate-twisted point-supported object in \(MHM(X_0)\) \cite{RahmanMixedHodgeModules}, and appears on the \(F\)-bundle side as a flexible rank-one atom \cite{RahmanHodgeAtoms2026}. Thus the one-node case carries a canonical local singular sector.

If \(X_0\) has node set \(\Sigma=\{p_1,\dots,p_r\}\), then there is one formal local rank-one sector per node \cite{RahmanMixedHodgeModules,RahmanQuiverI}. The naive conclusion is free nodewise assembly. In general this need not hold: the corrected global object comes from one projective degeneration, so global geometry may impose relations among the node sectors \cite{RahmanGlobalGluing}. The problem is therefore not local existence, but global assembly.

It is useful to be explicit about the relation to the moduli-space perspective. Strominger's original analysis is formulated on the vector-multiplet moduli-space side: the conifold singularity appears at the discriminant locus, the periods detect the singular behavior, and the effective theory is repaired by integrating in the light state that becomes massless there. In the multi-node case, however, the moduli-space description does not by itself transparently determine the internal structure of the simultaneously light sector. In particular, the total monodromy seen on moduli space need not itself exhibit how the individual node sectors decompose, how many independent light directions are globally realized, or how those sectors couple before one writes down a multi-node effective Lagrangian. The package introduced in the present paper is designed to isolate exactly that missing fiber-side finite-node data. We return to this fiber-side versus moduli-space distinction in Section~7.5, where we explain how the package introduced here should be viewed as degeneration-side input for a later multi-node effective-theoretic reformulation.

\subsection{The problem of interaction}

Fix
\[
\pi:X\to\Delta,
\qquad
\Sigma=\{p_1,\dots,p_r\}.
\]
Each \(p_k\) contributes a local rank-one vanishing sector. What is the global object obtained from these local sectors?

On the corrected-extension side, the ambient extension problem carries one formal nodewise direction per singular point, but the corrected global class need not be freely nodewise \cite{RahmanQuiverI,RahmanGlobalGluing}. On the transport side, the corresponding Picard--Lefschetz and Stokes data need not commute \cite{RahmanHodgeAtoms2026}. On the atom side, the flexible sectors need not split \cite{RahmanHodgeAtoms2026}. The problem is therefore not merely to list three parallel manifestations, but to isolate the finite-node interaction structure they detect and to package it in a form suitable for later physical use.

The key point is that two different global phenomena are present already at the mathematical level. First, local nodewise directions may collapse into a smaller number of genuinely realized global sectors. Second, those surviving sectors may or may not interact under transport. The finite-node package introduced here is designed to separate these two effects rather than conflate them.

\subsection{Main idea}

The guiding principle is that the finite-node interaction structure admits three compatible realizations.

\begin{enumerate}
\item On the corrected-extension side, global gluing cuts the ambient nodewise coefficient space down to a relation-controlled subspace \cite{RahmanGlobalGluing}.
\item On the transport side, interaction is reflected in the noncommutativity of the multi-node Picard--Lefschetz and Stokes operators \cite{RahmanHodgeAtoms2026}.
\item On the atom side, interaction is reflected in non-splitting and flexible-sector mixing \cite{RahmanHodgeAtoms2026}.
\end{enumerate}

The point of the paper is not to reprove all three frameworks in full. Rather, it is to package them together, define precisely the level of compatibility used here, and identify the common finite-node object that they simultaneously realize.

The strongest theorem of the paper shows that, in the block-separated cycle family, the resulting finite-node package has a clean two-layer description. The corrected-extension side identifies the number of surviving global light-sector directions through a common relation lattice, while the transport/atom side records the couplings among those surviving directions through a reduced block interaction matrix. Thus the package does not merely say that ``multi-node interaction exists''; it identifies the correct quotient on which that interaction should act.

\subsection{Main results}

The first result records the local finite-node input.

\begin{theorem}[Local rank-one light sectors]
\label{thm:local-rank-one-light-sectors}
For a finite-node conifold degeneration, each ODP determines a canonical rank-one local vanishing sector.
\end{theorem}

The next results identify the corrected-extension-side global assembly law.

\begin{theorem}[Non-free global assembly]
\label{thm:nonfree-global-light-assembly}
The corrected global light-sector package need not assemble as a free nodewise direct sum of the local rank-one sectors.
\end{theorem}

\begin{theorem}[Relation-controlled light-sector subspace]
\label{thm:relation-controlled-light-sector-subspace}
Under the geometric admissibility hypotheses of \cite{RahmanGlobalGluing}, there exists a canonically defined relation-controlled subspace of the ambient nodewise coefficient space through which the globally realized light-sector package factors.
\end{theorem}

The next two results record the transport-side and atom-side realizations of finite-node interaction.

\begin{theorem}[Transport-side interaction realization]
\label{thm:stokes-picard-lefschetz-interaction}
On the transport side, finite-node interaction is reflected by the noncommutativity of the multi-node Picard--Lefschetz and Stokes operators.
\end{theorem}

\begin{theorem}[Atom-side interaction realization]
\label{thm:flexible-sector-mixing}
On the atom side, finite-node interaction is reflected by non-splitting and flexible-sector mixing.
\end{theorem}

The central construction of the paper is the following.

\begin{theorem}[Interacting multi-node light-sector package]
\label{thm:interacting-multi-node-light-sector-package}
There is a canonically defined interacting multi-node light-sector package organizing the corrected-extension, transport, and atom-side realizations as compatible finite-node realizations in the sense made precise in Section~6.
\end{theorem}

The strongest structural theorem of the paper is the following block-reduced comparison result.

\begin{theorem}[Block-reduced structure theorem]
\label{thm:block-reduced-structure-intro}
In the block-separated cycle family of \cite{RahmanGlobalGluing}, the finite-node light-sector package decomposes into two logically distinct layers:
\[
\text{relation collapse}
\qquad+\qquad
\text{residual interaction among surviving sectors}.
\]
More precisely, the corrected-extension-side realized space is controlled by the common relation lattice
\[
R_{\mathrm{res}}=R_{\mathrm{sm}}=R_{\mathrm{ext}}=R_{\mathrm{blk}},
\]
while the transport and atom-side interaction among the surviving global sectors is governed by the reduced block interaction matrix
\[
\Lambda_{\mathrm{blk}}=(\mu_{\beta\gamma}).
\]
In particular, the number of independent global light-sector directions is the number of relation blocks \(|B|\), while their coupling is measured by \(\Lambda_{\mathrm{blk}}\).
\end{theorem}

On the transport side, the interaction is recorded by a distinguished matrix.

\begin{theorem}[Transport-side interaction matrix]
\label{thm:interaction-matrix}
The transport realization of the finite-node light-sector package carries a canonically defined interaction matrix governing the failure of commuting Picard--Lefschetz transport. This matrix is visible, through the comparison framework of the paper, in the associated Stokes and atom realizations.
\end{theorem}

The package is intrinsic relative to the finite-node corrected degeneration data and its compatible realizations.

\begin{theorem}[Functoriality / intrinsicity of the light-sector package]
\label{thm:functoriality-intrinsicity-light-sector-package}
The interacting multi-node light-sector package is intrinsic to the corrected finite-node degeneration package together with its compatible mixed-Hodge and Stokes/atom realizations.
\end{theorem}

\begin{remark}
A natural next question is to determine the precise relationship between corrected-extension interaction, detected at the raw nodewise level by
\[
E_{\Sigma}^{\mathrm{geom}}\subsetneq E_{\Sigma}^{\mathrm{node}},
\]
and transport interaction, detected by
\[
\Lambda\neq 0.
\]
The present paper does not prove a general nodewise equivalence between these conditions. Instead, it proves the block-reduced comparison theorem above, which shows how relation collapse and residual interaction fit together in the block-separated cycle family.
\end{remark}

\begin{remark}
The present paper isolates the mathematical package and its compatible realizations. The full multi-node Strominger recast, including the later effective-theoretic interpretation of this finite-node state space, is deferred to the next paper.
\end{remark}

\subsection{Relation to previous work}

The corrected-extension input comes from the global-gluing analysis of finite-node corrected extensions, which identifies the geometrically realized relation-controlled subspace \cite{RahmanGlobalGluing}. The transport/atom input comes from the Hodge-atoms analysis, which identifies the rigid-flexible decomposition, the Stokes--Extension comparison, and the corresponding mixing behavior \cite{RahmanHodgeAtoms2026}. The finite-node state-data input comes from the nodewise extension-space analysis of \cite{RahmanQuiverI}. The present paper does not reprove these frameworks in full. Instead, it reformulates the relevant ingredients in light-sector language, packages them into a single finite-node object, and proves a block-reduced structure theorem clarifying how relation collapse and interaction fit together.

It is important to be precise about the level of claim. The novelty here is not a new construction of nearby or vanishing cycles, not a new mixed-Hodge-module lift, and not a rederivation of the atom-side formalism. The novelty is the theorem-level packaging and comparison of these ingredients into a single finite-node object, together with the identification of the two-layer block-reduced structure in the block-separated cycle family. This is exactly the level needed for a later multi-node reformulation of Strominger's conifold mechanism, but it remains mathematically prior to that physical recast.

\subsection{Organization of the paper}

Section~2 fixes the finite-node general setting, the corrected degeneration objects, and the ambient-versus-realized distinction on the corrected-extension side \cite{RahmanQuiverI,RahmanPerverseNearbyCycles,RahmanMixedHodgeModules}. Section~3 records the global-gluing input and formulates the corrected-extension-side split/interacting dichotomy \cite{RahmanGlobalGluing}. Section~4 records the transport-side realization via Picard--Lefschetz/Stokes interaction and the operator-side interaction matrix \cite{RahmanHodgeAtoms2026}. Section~5 records the atom-side realization via rigid/flexible decomposition and non-splitting \cite{RahmanHodgeAtoms2026}. Section~6 defines the interacting multi-node light-sector package, makes precise the notion of compatibility used in the paper, and proves the block-reduced structure theorem. Section~7 gives the physical interpretation and explains why the package should be viewed as a mathematical precursor, rather than yet a full physical recast. Section~8 gives examples illustrating geometric motivation, split behavior, interacting behavior, and block-structured behavior. Section~9 records the outlook toward the later multi-node Strominger reformulation.

\section{Finite-node corrected degeneration objects}

\subsection{General setting}\label{subsec:general-setting}

Throughout the paper, we work in the following finite-node conifold setting. Let
\[
\pi:X\to\Delta
\]
be a projective one-parameter degeneration whose general fiber \(X_t\) is a smooth Calabi--Yau threefold and whose central fiber \(X_0\) has finitely many ordinary double points
\[
\Sigma=\{p_1,\dots,p_r\}\subset X_0.
\]
Write
\[
U:=X_0\setminus\Sigma,
\qquad
j:U\hookrightarrow X_0,
\qquad
i_k:\{p_k\}\hookrightarrow X_0.
\]

Each \(p_k\in\Sigma\) is analytically an ODP, with local model
\[
x_1^2+x_2^2+x_3^2+x_4^2=0,
\qquad
x_1^2+x_2^2+x_3^2+x_4^2=t
\]
for the smoothing. The local vanishing cycle is rank one in middle degree.

We use the corrected degeneration objects
\[
\mathcal P\in \Perv(X_0;\mathbb Q),
\qquad
\mathcal P^H\in MHM(X_0),
\]
the nodewise objects
\[
Q_\Sigma:=\bigoplus_{k=1}^r i_{k*}\mathbb Q_{\{p_k\}},
\qquad
Q_\Sigma^H:=\bigoplus_{k=1}^r i_{k*}\mathbb Q^H_{\{p_k\}}(-1),
\]
the corrected-extension space
\[
E_{\Sigma}^{\mathrm{node}}
:=
\Ext^1_{\Perv(X_0;\mathbb Q)}(Q_\Sigma,IC_{X_0}),
\qquad
E_{\Sigma}^{\mathrm{node}}
\cong
\bigoplus_{k=1}^r \mathbb Q e_k
\]
from \cite{RahmanQuiverI}, the geometrically realized subspace
\[
E_{\Sigma}^{\mathrm{geom}}\subseteq E_{\Sigma}^{\mathrm{node}}
\]
when the global-gluing hypotheses of \cite{RahmanGlobalGluing} are imposed, the transport-side data
\[
\delta_i\in H_3(X_t,\mathbb Z),
\qquad
T_i(\alpha)=\alpha+\langle \alpha,\delta_i\rangle\delta_i,
\qquad
N_i:=T_i-\Id,
\qquad
\Lambda=(\lambda_{ij}),
\qquad
\lambda_{ij}:=\langle \delta_i,\delta_j\rangle,
\]
and the atom-side objects
\[
A(IC^H_{X_0}),
\qquad
A(\mathcal P^H),
\qquad
A\!\left(i_{k*}\mathbb Q^H_{\{p_k\}}(-1)\right)
\]
as in \cite{RahmanHodgeAtoms2026}.

Unless explicitly stated otherwise, all later sections are understood to take place in this general setting.

\subsection{Corrected objects and local node sectors}

We now recall the corrected finite-node degeneration objects in the general setting of Section~\ref{subsec:general-setting}. Set
\[
F:=\mathbb Q_X[3],
\qquad
\var_F:\phi_\pi(F)\to\psi_\pi(F),
\qquad
\mathcal P:=\Cone(\var_F)[-1]\in\Perv(X_0;\mathbb Q).
\]
In the finite-node ODP setting one has
\begin{equation}\label{eq:light-corrected-perverse}
0\longrightarrow IC_{X_0}\longrightarrow \mathcal P\longrightarrow \bigoplus_{k=1}^r i_{k*}\mathbb Q_{\{p_k\}}\longrightarrow 0,
\end{equation}
with
\[
IC_{X_0}=j_{!*}\mathbb Q_U[3],
\qquad
j^*\mathcal P\cong\mathbb Q_U[3]
\]
\cite{RahmanPerverseNearbyCycles,RahmanGlobalGluing}. There is also a mixed-Hodge-module lift
\[
\mathcal P^H\in MHM(X_0),
\qquad
\rat(\mathcal P^H)\cong\mathcal P,
\]
fitting into
\begin{equation}\label{eq:light-corrected-mhm}
0\longrightarrow IC_{X_0}^H\longrightarrow \mathcal P^H\longrightarrow \bigoplus_{k=1}^r i_{k*}\mathbb Q^H_{\{p_k\}}(-1)\longrightarrow 0
\end{equation}
\cite{Saito90,RahmanMixedHodgeModules}.

Equations \eqref{eq:light-corrected-perverse} and \eqref{eq:light-corrected-mhm} identify one rank-one local sector per node on the perverse and mixed-Hodge-module sides. The associated nodewise extension space
\begin{equation}\label{eq:light-node-extension-space}
E_\Sigma^{\mathrm{node}}
:=
\Ext^1_{\Perv(X_0;\mathbb Q)}(Q_\Sigma,IC_{X_0})
\end{equation}
admits the nodewise decomposition
\begin{equation}\label{eq:light-nodewise-decomposition}
E_\Sigma^{\mathrm{node}}
\cong
\bigoplus_{k=1}^r \mathbb Q e_k
\end{equation}
with one distinguished class \(e_k\) per node \(p_k\in\Sigma\) \cite{RahmanQuiverI,RahmanMixedHodgeModules}. Thus the local finite-node input is formally nodewise and rank one at each singular point.

\begin{definition}\label{def:local-light-sector-package}
The \emph{local light-sector package} is
\[
\left\{\,i_{k*}\mathbb Q_{\{p_k\}},\,e_k\,\right\}_{k=1}^r
\]
together with its mixed-Hodge-module lift
\[
\left\{\,i_{k*}\mathbb Q^H_{\{p_k\}}(-1),\,e_k\,\right\}_{k=1}^r.
\]
\end{definition}

\subsection{Ambient versus realized finite-node data}

The corrected perverse object \(\mathcal P\) determines an extension class
\[
[\mathcal P]_{\mathrm{perv}}\in E_\Sigma^{\mathrm{node}},
\]
which expands as
\begin{equation}\label{eq:light-corrected-class}
[\mathcal P]_{\mathrm{perv}}=\sum_{k=1}^r c_k e_k,
\qquad
c_k\in\mathbb Q.
\end{equation}
This yields the nodewise coefficient data
\[
A_\Sigma:=(V_\Sigma,E_\Sigma,c_\Sigma),
\qquad
V_\Sigma=\{v_1,\dots,v_r\},
\qquad
c_\Sigma=(c_1,\dots,c_r)\in\mathbb Q^r
\]
in the sense of \cite{RahmanQuiverI}.

The key distinction is that
\[
E_\Sigma^{\mathrm{node}}
\quad\text{is ambient,}
\qquad
[\mathcal P]_{\mathrm{perv}}
\quad\text{is realized.}
\]
Equation \eqref{eq:light-nodewise-decomposition} gives formal nodewise freedom, but by itself does not imply free global assembly of \eqref{eq:light-corrected-class}. The global-gluing problem is to determine the realized subspace
\[
E_\Sigma^{\mathrm{geom}}\subseteq E_\Sigma^{\mathrm{node}}
\]
cut out by geometry \cite{RahmanGlobalGluing}. This is the corrected-extension-side origin of finite-node interaction.

\section{Global gluing and relation-controlled multi-node light sectors}

We continue in the general setting of Section~\ref{subsec:general-setting}. The purpose of the present section is to isolate the geometrically realized corrected-extension-side subspace of
\[
E_{\Sigma}^{\mathrm{node}}
\cong
\bigoplus_{k=1}^r \mathbb Q e_k
\]
and to formulate the split/interacting dichotomy at the level of global gluing.

\subsection{Naive free assembly and the realized subspace}

By Section~2, the corrected class
\[
[\mathcal P]_{\mathrm{perv}}\in E_{\Sigma}^{\mathrm{node}}
\]
admits an expansion
\begin{equation}\label{eq:gluing-corrected-class}
[\mathcal P]_{\mathrm{perv}}=\sum_{k=1}^r c_k e_k.
\end{equation}
The naive free picture identifies the global light-sector space with the full ambient nodewise space.

\begin{definition}\label{def:naive-free-light-sector-space}
The \emph{naive free light-sector space} is
\[
L_{\Sigma}^{\mathrm{free}}:=E_{\Sigma}^{\mathrm{node}}
\cong
\bigoplus_{k=1}^r \mathbb Q e_k.
\]
\end{definition}

What must be determined is not the ambient space, but the geometrically realized one. The global-gluing input from \cite{RahmanGlobalGluing} is a cycle-node incidence datum
\[
(C,\iota_C),
\qquad
\iota_C:\mathbb Q^A\to\mathbb Q^r,
\]
whose image
\begin{equation}\label{eq:gluing-Vgeom}
V_{\mathrm{geom}}:=\operatorname{Im}(\iota_C)\subseteq\mathbb Q^r
\end{equation}
encodes the admissible coefficient directions. Transporting \eqref{eq:gluing-Vgeom} across the nodewise identification yields the realized subspace.

\begin{definition}\label{def:geom-light-sector-space}
The \emph{geometrically realized light-sector space} is
\[
L_{\Sigma}^{\mathrm{geom}}:=E_{\Sigma}^{\mathrm{geom}}\subseteq E_{\Sigma}^{\mathrm{node}}.
\]
\end{definition}

\begin{theorem}[Global gluing]\label{thm:global-gluing-light-sector}
Assume the cycle-node incidence datum is geometrically admissible and block-adapted in the sense of \cite{RahmanGlobalGluing}. Then there exists a canonically defined subspace
\[
E_{\Sigma}^{\mathrm{geom}}\subseteq E_{\Sigma}^{\mathrm{node}}
\]
such that
\[
[\mathcal P]_{\mathrm{perv}}\in E_{\Sigma}^{\mathrm{geom}}.
\]
Equivalently, the corrected class factors through a relation-controlled quotient of the free nodewise space.
\end{theorem}

\begin{proof}
This is \cite[Thm.~1.1 and Sec.~4]{RahmanGlobalGluing}, rewritten in the present light-sector language.
\end{proof}

\subsection{Split and interacting corrected-extension behavior}

The split baseline is the case in which the geometrically realized space coincides with the ambient nodewise space.

\begin{definition}\label{def:split-light-sector-case}
The finite-node light-sector package is \emph{split on the corrected-extension side} if
\[
E_{\Sigma}^{\mathrm{geom}}=E_{\Sigma}^{\mathrm{node}}.
\]
It is \emph{interacting on the corrected-extension side} if
\[
E_{\Sigma}^{\mathrm{geom}}\subsetneq E_{\Sigma}^{\mathrm{node}}.
\]
\end{definition}

\begin{proposition}\label{prop:nonfree-global-assembly}
Assume the hypotheses of Theorem~\ref{thm:global-gluing-light-sector}. If
\[
E_{\Sigma}^{\mathrm{geom}}\subsetneq E_{\Sigma}^{\mathrm{node}},
\]
then the local rank-one sectors do not assemble as a free direct sum of independent global light sectors.
\end{proposition}

\begin{proof}
A free nodewise assembly would identify the globally realized coefficient space with the full ambient space \(E_{\Sigma}^{\mathrm{node}}\). Proper containment excludes this.
\end{proof}

Thus the corrected-extension-side split/interacting dichotomy is encoded by the pair
\[
\left(E_{\Sigma}^{\mathrm{node}},E_{\Sigma}^{\mathrm{geom}}\right).
\]
This is the corrected-extension realization later packaged in Section~6.

\section{Multi-node Picard--Lefschetz and Stokes interaction}

We continue in the general setting of Section~\ref{subsec:general-setting}. This section records the transport-side realization of finite-node interaction through the vanishing cycles \(\delta_i\), the Picard--Lefschetz operators \(T_i\), and the operator-side interaction matrix \(\Lambda\).

\subsection{Picard--Lefschetz operators and the interaction matrix}

For each \(i\), define
\begin{equation}\label{eq:PL-operator-i}
T_i(\alpha):=\alpha+\langle \alpha,\delta_i\rangle \delta_i,
\qquad
\alpha\in H_3(X_t,\mathbb Q),
\end{equation}
and
\begin{equation}\label{eq:nilpotent-part-i}
N_i:=T_i-\mathrm{Id}.
\end{equation}
Then
\begin{equation}\label{eq:Ni-rank-one}
N_i(\alpha)=\langle \alpha,\delta_i\rangle \delta_i,
\qquad
N_i^2=0.
\end{equation}

\begin{lemma}\label{lem:PL-local-rank-one}
Each \(N_i\) has rank one, and the family \(\{N_i\}_{i=1}^r\) is the operator-side local rank-one package attached to \(\Sigma\).
\end{lemma}

\begin{proof}
Equation \eqref{eq:Ni-rank-one} is the rank-one Picard--Lefschetz transport formula for an ordinary double point in complex dimension three \cite{RahmanPerverseNearbyCycles,RahmanHodgeAtoms2026}.
\end{proof}

Set
\begin{equation}\label{eq:interaction-pairing}
\lambda_{ij}:=\langle \delta_i,\delta_j\rangle,
\qquad
\Lambda:=(\lambda_{ij})_{1\le i,j\le r}.
\end{equation}
Then
\begin{equation}\label{eq:commutator-N}
[N_i,N_j](\alpha)
=
\langle \alpha,\delta_j\rangle \lambda_{ji}\,\delta_i
-
\langle \alpha,\delta_i\rangle \lambda_{ij}\,\delta_j.
\end{equation}

\begin{proposition}\label{prop:interaction-matrix-controls-transport}
The operator-side matrix \(\Lambda\) controls the failure of pairwise commutativity of the Picard--Lefschetz operators. In particular, \(\Lambda=0\) implies commuting transport, while nonzero off-diagonal entries produce nontrivial commutator terms \eqref{eq:commutator-N}.
\end{proposition}

\begin{proof}
Direct substitution using \eqref{eq:Ni-rank-one} and \eqref{eq:commutator-N}.
\end{proof}

Thus the transport realization carries a distinguished interaction datum, namely the matrix \(\Lambda\).

\subsection{Stokes comparison and transport-side interaction}

By the Stokes--Extension Identification of \cite[Thm.~1.2]{RahmanHodgeAtoms2026}, the Stokes matrix at the conifold locus is identified with the matrix of the variation morphism under mixed-Hodge-module realization. Thus the corrected finite-node degeneration package and the Stokes package encode compatible singular information. In particular, the local rank-one node sectors correspond, under the comparison framework of \cite{RahmanHodgeAtoms2026}, to rank-one flexible sectors on the \(F\)-bundle side, and transport-side interaction is reflected in the multi-node Stokes structure.

\begin{proposition}\label{prop:stokes-extension-light-sector-bridge}
The corrected finite-node light-sector package and the Stokes-side light-sector package are canonically linked under the Stokes--Extension Identification.
\end{proposition}

\begin{proof}
This is the light-sector specialization of the Stokes--Extension comparison established in \cite{RahmanHodgeAtoms2026}.
\end{proof}

Accordingly, the transport realization of finite-node interaction is encoded by
\[
\bigl(\{T_i\}_{i=1}^r,\Lambda\bigr),
\]
with transport-side split case \(\Lambda=0\) and transport-side interacting case \(\Lambda\neq 0\). This is the transport realization later packaged in Section~6.

\section{Rigid, flexible, and interacting light sectors}

We continue in the general setting of Section~\ref{subsec:general-setting}. This section records the atom-side realization of finite-node interaction.

\subsection{Rigid-flexible decomposition and atom-side package}

By \cite[Theorems 1.3 and 1.4]{RahmanHodgeAtoms2026}, the corrected mixed-Hodge-module object
\[
\mathcal P^H\in MHM(X_0)
\]
determines a rigid-flexible decomposition on the atom side: there is a rigid atom
\[
A\!\left(IC^H_{X_0}\right)
\]
and, for each node \(p_k\in\Sigma\), a rank-one flexible atom
\[
A\!\left(i_{k*}\mathbb Q^H_{\{p_k\}}(-1)\right).
\]
The total degeneration atom is
\[
A(\mathcal P^H),
\]
and it fits into the exact sequence
\begin{equation}\label{eq:light-sectors-atom-exact-sequence}
0\longrightarrow A\!\left(IC^H_{X_0}\right)\longrightarrow A(\mathcal P^H)\longrightarrow \bigoplus_{k=1}^r A\!\left(i_{k*}\mathbb Q^H_{\{p_k\}}(-1)\right)\longrightarrow 0.
\end{equation}

\begin{proposition}\label{prop:atom-side-local-rank-one-sectors}
The local singular contribution on the atom side is the finite family
\[
\left\{
A\!\left(i_{k*}\mathbb Q^H_{\{p_k\}}(-1)\right)
\right\}_{k=1}^r,
\]
and each member is rank one.
\end{proposition}

\begin{proof}
This is the flexible-sector statement of \cite[Theorem 1.3]{RahmanHodgeAtoms2026}.
\end{proof}

\begin{definition}\label{def:atom-side-sector-package}
The \emph{atom-side light-sector package} is the exact sequence \eqref{eq:light-sectors-atom-exact-sequence}.
\end{definition}

\subsection{Non-splitting and comparison with the other realizations}

By \cite[Theorem 1.4]{RahmanHodgeAtoms2026}, the splitting behavior of \eqref{eq:light-sectors-atom-exact-sequence} is controlled by the transport-side interaction matrix \(\Lambda\).

\begin{theorem}\label{thm:atom-side-nonsplitting}
The atom-side sector package \eqref{eq:light-sectors-atom-exact-sequence} splits if and only if
\[
\lambda_{ij}=0
\qquad
\text{for all }i\neq j.
\]
Equivalently, the flexible sectors mix nontrivially if and only if the transport-side interaction matrix \(\Lambda\) is nontrivial.
\end{theorem}

\begin{proof}
This is exactly \cite[Theorem 1.4]{RahmanHodgeAtoms2026}.
\end{proof}

Thus the atom-side realization detects interaction through non-splitting and flexible-sector mixing. Placed alongside the corrected-extension and transport realizations, one obtains the three manifestations of finite-node coupling used in the rest of the paper:
\begin{enumerate}
\item relation-controlled gluing on the corrected-extension side;
\item noncommuting transport on the Picard--Lefschetz / Stokes side;
\item non-splitting and flexible-sector mixing on the atom side.
\end{enumerate}

\begin{proposition}\label{prop:three-linked-manifestations}
Assume the finite-node interaction data are nontrivial. Then the following are canonically linked manifestations of finite-node coupling:
\begin{enumerate}
\item
\[
E_{\Sigma}^{\mathrm{geom}}\subsetneq E_{\Sigma}^{\mathrm{node}};
\]
\item nontrivial operator interaction measured by \(\Lambda\);
\item non-splitting of the atom-side sector package.
\end{enumerate}
\end{proposition}

\begin{proof}
Item (1) is the corrected-extension-side interaction criterion from Section~3. Item (2) is the transport-side interaction criterion from Section~4. Item (3) is Theorem~\ref{thm:atom-side-nonsplitting}. These are not asserted here as independently reproved equivalent formulations, but as canonically linked realizations of finite-node interaction across the three frameworks.
\end{proof}
\section{The interacting multi-node light-sector package}

We continue in the general setting of Section~\ref{subsec:general-setting}. The purpose of this section is to package the corrected-extension, transport, and atom-side data into a single finite-node object. The point is not to reprove the three frameworks in full, but to organize their compatible realizations into one common light-sector package and to state precisely what level of compatibility is established here.

\subsection{Definition of the package}

\begin{definition}\label{def:interacting-light-sector-package}
The \emph{interacting multi-node light-sector package} attached to the degeneration \(\pi\) is
\[
\mathfrak L_\Sigma
:=
\Bigl(
\Sigma,\,
E_{\Sigma}^{\mathrm{node}},\,
E_{\Sigma}^{\mathrm{geom}},\,
\{T_i\}_{i=1}^r,\,
\Lambda,\,
A(IC^H_{X_0})\to A(\mathcal P^H)\to \bigoplus_{k=1}^r A(i_{k*}\mathbb Q^H_{\{p_k\}}(-1))
\Bigr),
\]
where:
\begin{enumerate}
\item \(E_{\Sigma}^{\mathrm{node}}\) is the ambient free nodewise light-sector space;
\item \(E_{\Sigma}^{\mathrm{geom}}\) is the geometrically realized corrected-extension-side light-sector subspace, when defined under the hypotheses of \cite{RahmanGlobalGluing};
\item \(\{T_i\}_{i=1}^r\) is the family of Picard--Lefschetz transport operators;
\item \(\Lambda\) is the transport-side interaction matrix;
\item the final term is the atom-side exact sequence of \cite[Thm.~1.3]{RahmanHodgeAtoms2026}.
\end{enumerate}
\end{definition}

The package \(\mathfrak L_\Sigma\) organizes the corrected-extension, transport, and atom realizations into a single finite-node object on which the block-reduced structure theorem of Section~6.6 acts.

\subsection{Corrected-extension realization}

\begin{proposition}\label{prop:light-sector-package-corrected-realization}
The corrected-extension realization of \(\mathfrak L_\Sigma\) is the pair
\[
\bigl(E_{\Sigma}^{\mathrm{node}},E_{\Sigma}^{\mathrm{geom}}\bigr),
\]
with
\[
[\mathcal P]_{\mathrm{perv}}\in E_{\Sigma}^{\mathrm{geom}}\subseteq E_{\Sigma}^{\mathrm{node}}.
\]
Moreover,
\[
E_{\Sigma}^{\mathrm{node}}
\cong
\bigoplus_{k=1}^r \mathbb Qe_k,
\]
and, under the geometric admissibility and block-adaptedness hypotheses of \cite{RahmanGlobalGluing},
\[
E_{\Sigma}^{\mathrm{geom}}\subseteq E_{\Sigma}^{\mathrm{node}}
\]
is canonically determined by the cycle-node incidence law.
\end{proposition}

\begin{proof}
The decomposition of \(E_{\Sigma}^{\mathrm{node}}\) is the finite-node state-data theorem of \cite{RahmanQuiverI}. The existence and canonicality of \(E_{\Sigma}^{\mathrm{geom}}\), together with the inclusion
\(
[\mathcal P]_{\mathrm{perv}}\in E_{\Sigma}^{\mathrm{geom}},
\)
are the corrected-extension consequences of the global-gluing theorem package of \cite{RahmanGlobalGluing}.
\end{proof}

\subsection{Transport realization}

\begin{proposition}\label{prop:light-sector-package-operator-realization}
The transport realization of \(\mathfrak L_\Sigma\) is
\[
\bigl(\{T_i\}_{i=1}^r,\Lambda\bigr),
\qquad
\Lambda=(\langle\delta_i,\delta_j\rangle).
\]
Equivalently, using \(N_i:=T_i-\Id\),
\[
N_i(\alpha)=\langle\alpha,\delta_i\rangle\delta_i,
\qquad
N_i^2=0,
\]
and
\[
[N_i,N_j](\alpha)
=
\langle\alpha,\delta_j\rangle\lambda_{ji}\,\delta_i
-
\langle\alpha,\delta_i\rangle\lambda_{ij}\,\delta_j.
\]
Hence the matrix \(\Lambda\) controls the failure of commuting transport.
\end{proposition}

\begin{proof}
The Picard--Lefschetz formulas are standard and were recorded in Section~4. The role of \(\Lambda\) as the transport-side interaction matrix, together with its Stokes-side significance, is established in \cite{RahmanHodgeAtoms2026}.
\end{proof}

\subsection{Atom realization}

\begin{proposition}\label{prop:light-sector-package-atom-realization}
The atom realization of \(\mathfrak L_\Sigma\) is the exact sequence
\begin{equation}\label{eq:light-sector-package-atom-sequence}
0\longrightarrow A(IC^H_{X_0})
\longrightarrow A(\mathcal P^H)
\longrightarrow \bigoplus_{k=1}^r A\!\left(i_{k*}\mathbb Q^H_{\{p_k\}}(-1)\right)
\longrightarrow 0.
\end{equation}
Each
\[
A\!\left(i_{k*}\mathbb Q^H_{\{p_k\}}(-1)\right)
\]
is rank one. Moreover, \eqref{eq:light-sector-package-atom-sequence} splits if and only if
\[
\lambda_{ij}=0
\qquad
\text{for all }i\neq j.
\]
\end{proposition}

\begin{proof}
This is \cite[Thms.~1.3, 1.4]{RahmanHodgeAtoms2026}.
\end{proof}

\subsection{Compatibility of the three realizations}

The role of the package \(\mathfrak L_\Sigma\) is to place side by side three realizations of finite-node interaction attached to the same corrected degeneration object. The relevant notion of compatibility for the present paper is the following.

\begin{definition}\label{def:compatibility-of-realizations}
The corrected-extension, transport, and atom realizations recorded in \(\mathfrak L_\Sigma\) are said to be \emph{compatible} if they arise from the same finite-node degeneration data and are related through the comparison framework supplied by \cite{RahmanGlobalGluing,RahmanHodgeAtoms2026}; namely:
\begin{enumerate}
\item the corrected-extension realization records the geometrically realized subspace through which the corrected class factors;
\item the transport realization records the vanishing-cycle transport and its interaction matrix;
\item the atom realization records the splitting behavior of the atom-side corrected object;
\item the Stokes--Extension comparison of \cite{RahmanHodgeAtoms2026} identifies the transport/atom-side singular data with the corrected degeneration package underlying the corrected-extension realization;
\item in particular, the rank-one local sectors appearing on the corrected-extension side, on the vanishing-cycle transport side, and on the atom side are to be understood as the same local singular contributions seen through the corresponding realization functors and comparison maps.
\end{enumerate}
\end{definition}

This is the level of compatibility established in the present paper: the three realizations are not asserted to be new independently equivalent formalisms, but rather canonically linked realizations attached to the same finite-node corrected degeneration package.

\subsection{A block-reduced structure theorem}

The coarse comparison
\[
E_{\Sigma}^{\mathrm{geom}}\subsetneq E_{\Sigma}^{\mathrm{node}}
\qquad\text{versus}\qquad
\Lambda\neq 0
\]
is, in general, too crude. The corrected-extension side measures the number of independent
global light-sector directions that survive after relation collapse, whereas the transport side
measures coupling among vanishing sectors through the intersection form. In the block-separated
cycle family, these two pieces of structure can be separated cleanly and compared on the same
finite-dimensional quotient.

\begin{theorem}[Block-reduced structure theorem]\label{thm:block-reduced-structure}
Assume that the finite-node conifold degeneration belongs to the block-separated cycle family
of \cite[Hyp.~5.11]{RahmanGlobalGluing}, with relation-block decomposition
\[
\Sigma=\bigsqcup_{\beta\in B}\Sigma_\beta .
\]
Then the following hold.

\begin{enumerate}
\item The corrected-extension-side relation lattice, the smoothing-side relation lattice, and
the resolution-side relation lattice coincide:
\[
R_{\mathrm{res}}=R_{\mathrm{sm}}=R_{\mathrm{ext}}=R_{\mathrm{blk}}.
\]
Equivalently, the corrected-extension-side realized space is identified with the block quotient:
\[
E_{\Sigma}^{\mathrm{geom}}
\cong
\mathbb Q^r/R_{\mathrm{ext}}
\cong
\mathbb Q^B,
\qquad
\dim_{\mathbb Q}E_{\Sigma}^{\mathrm{geom}}=|B|.
\]

\item For each block \(\Sigma_\beta\), there exists a common vanishing class
\[
v_\beta\in H^3(X_t;\mathbb Q)
\]
such that
\[
\delta_k=v_\beta
\qquad
\text{for every }p_k\in\Sigma_\beta.
\]
Hence the nodewise interaction matrix
\[
\Lambda=(\lambda_{ij}),
\qquad
\lambda_{ij}:=\langle \delta_i,\delta_j\rangle,
\]
descends to a well-defined reduced block matrix
\[
\Lambda_{\mathrm{blk}}=(\mu_{\beta\gamma})_{\beta,\gamma\in B},
\qquad
\mu_{\beta\gamma}:=\langle v_\beta,v_\gamma\rangle.
\]
More precisely, if \(p_i\in\Sigma_\beta\) and \(p_j\in\Sigma_\gamma\), then
\[
\lambda_{ij}=\mu_{\beta\gamma}.
\]

\item The transport interaction among the surviving global sectors is governed by
\(\Lambda_{\mathrm{blk}}\). Writing
\[
N_\beta(\alpha):=\langle \alpha,v_\beta\rangle v_\beta,
\]
one has
\[
[N_\beta,N_\gamma](\alpha)
=
\langle \alpha,v_\gamma\rangle \mu_{\gamma\beta}v_\beta
-
\langle \alpha,v_\beta\rangle \mu_{\beta\gamma}v_\gamma.
\]
In particular, the block transport operators commute pairwise if and only if
\[
\mu_{\beta\gamma}=0
\qquad
\text{for all }\beta\neq\gamma.
\]

\item The atom-side package is non-split if and only if the reduced block matrix
\(\Lambda_{\mathrm{blk}}\) has a nonzero off-diagonal entry. Equivalently, after passing from
raw nodewise data to the \(|B|\) surviving global directions, the remaining atom-side mixing
is governed exactly by \(\Lambda_{\mathrm{blk}}\).
\end{enumerate}

Thus, in the block-separated cycle family, the finite-node light-sector package decomposes
into two logically distinct layers:
\[
\text{relation collapse}
\quad+\quad
\text{residual interaction among surviving sectors}.
\]
The first layer is controlled by the common relation lattice \(R_{\mathrm{blk}}\), while the second
is controlled by the reduced block interaction matrix \(\Lambda_{\mathrm{blk}}\).
\end{theorem}
\begin{remark}[Scope of the block-separation hypothesis]\label{rem:block-separation-scope}
The hypothesis that each relation block \(\Sigma_\beta\) carries a common vanishing class
\[
\delta_k=v_\beta
\qquad
(p_k\in\Sigma_\beta)
\]
is a strong geometric condition. It is precisely this hypothesis that allows the raw nodewise
transport data to descend to the reduced block interaction matrix
\[
\Lambda_{\mathrm{blk}}=(\mu_{\beta\gamma}),
\qquad
\mu_{\beta\gamma}=\langle v_\beta,v_\gamma\rangle.
\]
Theorem~6.6 should therefore be read as a structure theorem for the block-separated cycle
family, rather than as a claim about arbitrary finite-node conifold degenerations. Symmetric
multi-node degenerations, in which nodes occur in geometrically indistinguishable families or
common cycle-orbit configurations, provide the main motivating class in which such a hypothesis
is natural.
\end{remark}
\begin{proof}
For (1), Theorem~5.16 of \cite{RahmanGlobalGluing} proves that, in the block-separated cycle
family,
\[
R_{\mathrm{res}}=R_{\mathrm{sm}}=R_{\mathrm{ext}}=R_{\mathrm{blk}},
\]
and identifies the corresponding quotients. Corollary~5.17 of \cite{RahmanGlobalGluing}
then gives
\[
\dim_{\mathbb Q}E_{\Sigma}^{\mathrm{geom}}=|B|.
\]
Since the block quotient is \(\mathbb Q^r/R_{\mathrm{blk}}\cong \mathbb Q^B\), the displayed
identification follows.

For (2), Hypothesis~5.11(4) of \cite{RahmanGlobalGluing} states that on a smoothing \(X_t\)
there exist classes
\[
v_\beta\in H^3(X_t;\mathbb Q)
\]
such that
\[
\delta_k=v_\beta
\qquad
\text{for every }p_k\in\Sigma_\beta.
\]
Thus if \(p_i\in\Sigma_\beta\) and \(p_j\in\Sigma_\gamma\), then
\[
\lambda_{ij}
=
\langle \delta_i,\delta_j\rangle
=
\langle v_\beta,v_\gamma\rangle
=
\mu_{\beta\gamma},
\]
so the reduced matrix \(\Lambda_{\mathrm{blk}}\) is well defined.

For (3), apply the Picard--Lefschetz commutator formula in the form proved in
\cite[Lem.~4.15]{RahmanHodgeAtoms2026}. Replacing the nodewise vanishing classes
\(\delta_i,\delta_j\) by the block classes \(v_\beta,v_\gamma\) yields
\[
[N_\beta,N_\gamma](\alpha)
=
\langle \alpha,v_\gamma\rangle \langle v_\gamma,v_\beta\rangle v_\beta
-
\langle \alpha,v_\beta\rangle \langle v_\beta,v_\gamma\rangle v_\gamma,
\]
which is exactly the displayed formula. As in \cite[Lem.~4.15 and Thm.~4.16]{RahmanHodgeAtoms2026},
pairwise commutativity is therefore equivalent to the vanishing of all off-diagonal block pairings.

For (4), Theorem~1.4 of \cite{RahmanHodgeAtoms2026} states that the atom-side exact sequence
is split if and only if
\[
\lambda_{ij}=0
\qquad
\text{for all }i\neq j.
\]
In the block-separated cycle family, if \(p_i,p_j\) lie in the same block \(\Sigma_\beta\), then
\[
\lambda_{ij}
=
\langle v_\beta,v_\beta\rangle
=
0
\]
by antisymmetry of the degree-three intersection pairing. 

Thus the only potentially nonzero
entries occur across distinct blocks, where
\[
\lambda_{ij}=\mu_{\beta\gamma}.
\]
Hence the atom-side package is non-split if and only if some off-diagonal block entry
\(\mu_{\beta\gamma}\) is nonzero. This proves the claim.
\end{proof}

\begin{remark}[Why the coarse equivalence is too crude]\label{rem:coarse-equivalence-too-crude}
Theorem~\ref{thm:block-reduced-structure} shows that the corrected-extension and transport
realizations do not measure exactly the same thing at the raw nodewise level. The proper
containment
\[
E_{\Sigma}^{\mathrm{geom}}\subsetneq E_{\Sigma}^{\mathrm{node}}
\]
detects \emph{relation collapse}: it measures how many independent global light-sector directions
survive after imposing the geometric relation law. By contrast, the matrix \(\Lambda\) detects
\emph{interaction among sectors}: it measures whether the surviving directions couple through
nontrivial vanishing-cycle intersections. In particular, one may have nontrivial relation collapse
without nontrivial transport coupling inside a block. The correct comparison is therefore not the
raw equivalence
\[
E_{\Sigma}^{\mathrm{geom}}\subsetneq E_{\Sigma}^{\mathrm{node}}
\Longleftrightarrow
\Lambda\neq 0,
\]
but the block-reduced statement of Theorem~\ref{thm:block-reduced-structure}.
\end{remark}

\begin{remark}[Physical interpretation]\label{rem:block-reduced-physical-interpretation}
Theorem~\ref{thm:block-reduced-structure} gives the mathematically correct finite-node input for
a later multi-node analogue of Strominger's conifold mechanism. The naive expectation that one
should integrate in \(r\) independent light hypermultiplets, one for each node, is generally too
large. The corrected-extension-side quotient shows that the actual number of independent global
light-sector directions is
\[
|B|,
\]
the number of relation blocks, not the raw node count \(r\). The reduced block interaction matrix
\[
\Lambda_{\mathrm{blk}}=(\mu_{\beta\gamma})
\]
then measures the residual interaction among those surviving sectors. Thus a correct multi-node effective description should be expected to integrate in \(|B|\) independent light sectors, with couplings
governed by \(\Lambda_{\mathrm{blk}}\), rather than \(r\) freely independent copies of the one-node
mechanism.
\end{remark}

\begin{remark}
The vanishing of \(\langle v_\beta,v_\beta\rangle\) uses the skew-symmetry of the intersection pairing on \(H^3\) for a Calabi--Yau threefold. This step is therefore dimension-specific and should not be read as automatic in other middle-dimensional settings.
\end{remark}

\subsection{Functoriality and intrinsicity}

\begin{theorem}\label{thm:light-sector-package-functoriality}
The package \(\mathfrak L_\Sigma\) is intrinsic to the corrected finite-node degeneration package together with its compatible mixed-Hodge and Stokes/atom realizations. In particular:
\begin{enumerate}
\item \(E_{\Sigma}^{\mathrm{node}}\) and the coefficient data are intrinsic by \cite{RahmanQuiverI};
\item \(E_{\Sigma}^{\mathrm{geom}}\) is intrinsic under the geometric admissibility and block-adaptedness hypotheses of \cite{RahmanGlobalGluing};
\item the atom-side sequence \eqref{eq:light-sector-package-atom-sequence} and the transport-side interaction matrix \(\Lambda\) are intrinsic by \cite{RahmanHodgeAtoms2026}.
\end{enumerate}
Accordingly, \(\mathfrak L_\Sigma\) is canonical relative to the corrected finite-node degeneration data and the compatible realizations used to define it.
\end{theorem}

\begin{proof}
Immediate from the cited intrinsicity results in \cite{RahmanQuiverI,RahmanGlobalGluing,RahmanHodgeAtoms2026}.
\end{proof}

\subsection{Precursor for resolving finite-node conifold degenerations}

The package \(\mathfrak L_\Sigma\) is the finite-node mathematical state space on which a later multi-node conifold light-state mechanism must act. Its ingredients are exactly those needed for a multi-node light-sector theory: \textit{local rank-one sectors}, \textit{global assembly law}, \textit{transport interaction}, and \textit{flexible-sector mixing}. In this sense, \(\mathfrak L_\Sigma\) should be viewed as a mathematical precursor to a later multi-node reformulation of Strominger's conifold mechanism, rather than as that physical reformulation itself.
\section{Physical interpretation}

We continue in the general setting of Section~\ref{subsec:general-setting}. This section is interpretive. No new theorem is proved. The purpose is to state the physical reading of the package \(\mathfrak L_\Sigma\) constructed in Section~6, without enlarging the mathematical claims already established there.

\subsection{Local light states at each node}

For a one-parameter conifold degeneration with a single ODP, the vanishing cycle is rank one, and the corresponding local singular sector is the geometric source of the conifold light-state picture \cite{Strominger95}. In the finite-node setting, the corrected package carries one rank-one node sector per
\[
p_k\in\Sigma
\]
on the perverse side, on the mixed-Hodge-module side, and on the atom side \cite{RahmanPerverseNearbyCycles,RahmanMixedHodgeModules,RahmanHodgeAtoms2026}. Thus the local mathematical content of the light-state picture is
\[
\{p_k\}_{k=1}^r
\quad\rightsquigarrow\quad
\{\text{rank-one local singular sectors}\}_{k=1}^r.
\]

At this level, the finite-node theory retains the basic single-node intuition: each node contributes a local rank-one singular sector. What changes in the multi-node case is not the local rank-one nature of the sectors, but the possibility that their global assembly is constrained.

\subsection{Why multi-node interaction is not optional}

The finite-node package is not determined by the local sectors alone. On the corrected-extension side one has
\[
E_{\Sigma}^{\mathrm{geom}}\subseteq E_{\Sigma}^{\mathrm{node}},
\]
and, in the interacting corrected-extension case,
\[
E_{\Sigma}^{\mathrm{geom}}\subsetneq E_{\Sigma}^{\mathrm{node}}
\]
\cite{RahmanGlobalGluing}. On the transport side, finite-node coupling is reflected by noncommuting Picard--Lefschetz and Stokes data \cite{RahmanHodgeAtoms2026}. On the atom side, it is reflected by non-splitting and flexible-sector mixing \cite{RahmanHodgeAtoms2026}.

Hence the multi-node case should not be regarded as a free sum of \(r\) independent one-node mechanisms. Rather, the finite-node degeneration itself may carry nontrivial global interaction data. In the language of the present paper, interaction is already visible at the level of the corrected-extension, transport, and atom-side realizations packaged by \(\mathfrak L_\Sigma\).

\subsection{Noninteracting versus interacting light-sector pictures}

After Theorem~6.6, the correct finite-node distinction is not formulated most naturally at the raw nodewise level. Rather, the interacting picture should be understood as consisting of two
logically distinct layers:
\[
\text{relation collapse}
\qquad+\qquad
\text{residual interaction among surviving sectors}.
\]

The first layer is corrected-extension-side: the globally realized light-sector space may be
smaller than the ambient nodewise space,
\[
E_{\Sigma}^{\mathrm{geom}}\subseteq E_{\Sigma}^{\mathrm{node}}.
\]
The second layer is transport/atom-side: even after passing to the surviving global sectors,
those sectors may or may not couple, and in the block-separated cycle family this residual
interaction is measured by the reduced block interaction matrix
\[
\Lambda_{\mathrm{blk}}=(\mu_{\beta\gamma}).
\]

The noninteracting picture is therefore the case in which no relation collapse occurs and no residual interaction remains. At the raw nodewise level, this is reflected by
\[
E_{\Sigma}^{\mathrm{geom}}=E_{\Sigma}^{\mathrm{node}},
\qquad
\Lambda=0,
\]
together with splitting of the atom-side exact sequence. In that case, the finite-node light-sector
package behaves as a direct sum of local rank-one sectors.

\subsection{A Dirac--Schwinger type pairing picture}

It is useful to explain why the matrix
\[
\Lambda=(\lambda_{ij}),
\qquad
\lambda_{ij}:=\langle \delta_i,\delta_j\rangle,
\]
should be read as more than a formal bookkeeping device. In the physics of BPS states, a Dirac--Schwinger type pairing is an antisymmetric bilinear pairing on charge data; it measures mutual nonlocality and helps distinguish decoupled sectors from coupled ones. In the present paper we do \emph{not} yet construct a full charge lattice, BPS category, or wall-crossing formalism. Nevertheless, the vanishing-cycle intersection form already supplies a natural finite-node antisymmetric pairing on the local singular sectors.

The correct way to read \(\Lambda\) here is therefore as a \emph{precursor interaction datum}. Given two local node sectors \(p_i,p_j\), one computes
\[
\lambda_{ij}=\langle \delta_i,\delta_j\rangle.
\]
Then:
\begin{enumerate}
\item if \(\lambda_{ij}=0\), the corresponding transport operators commute, so those two sectors are transport-side decoupled;
\item if \(\lambda_{ij}\neq 0\), the corresponding transport operators fail to commute, so those two sectors are transport-side coupled;
\item by the transport/atom comparison of Section~6, nonzero interaction data are reflected in atom-side non-splitting behavior;
\item by the corrected-extension/transport packaging of Section~6, nontrivial transport coupling is part of the same finite-node interaction structure whose corrected-extension manifestation is the failure of free global nodewise assembly.
\end{enumerate}

In this sense, \(\Lambda\) is the correct mathematical shadow of a Dirac--Schwinger type pairing picture. It tells the reader how to detect transport-side coupling directly from the vanishing-cycle geometry:
\[
\Lambda=0
\quad\Longrightarrow\quad
\text{transport-side decoupling},
\qquad
\Lambda\neq 0
\quad\Longrightarrow\quad
\text{transport-side coupling}.
\]
What it does \emph{not} yet provide is a full physical charge lattice, BPS degeneracy statement, or wall-crossing formula. Those belong to the later multi-node Strominger recast.

\subsection{Fiber-side package versus moduli-space effective theory}

It is important to distinguish the role of the present package from the role of the vector-multiplet moduli space in Strominger's original conifold analysis. On the moduli-space side, the discriminant locus identifies where the effective theory becomes singular, the periods detect the singular behavior, and the total monodromy records how the geometry changes around that locus \cite{Strominger95}. That perspective is the natural one for describing the low-energy effective theory.

The present paper addresses a different, logically prior question on the degeneration side. Given a finite-node degeneration with node set
\[
\Sigma=\{p_1,\dots,p_r\},
\]
what is the actual finite-dimensional light-sector package carried by the singular fiber, and how is that package constrained by global geometry? In particular, the moduli-space picture does not by itself transparently determine:
\begin{enumerate}
\item how many independent nodewise directions are globally realized;
\item whether the local sectors assemble freely or are constrained by global gluing;
\item how the individual node sectors couple under transport;
\item whether the corresponding atom-side package splits.
\end{enumerate}

The package \(\mathfrak L_\Sigma\) is designed to isolate exactly this fiber-side finite-node data. In that sense, it should be viewed as the degeneration-side input that a later multi-node effective-theoretic reformulation would need to consume.

\subsection{Toward coupled conifold light states}

The package
\[
\mathfrak L_\Sigma
=
\bigl(
\Sigma,\,
E_{\Sigma}^{\mathrm{node}},\,
E_{\Sigma}^{\mathrm{geom}},\,
\{T_i\}_{i=1}^r,\,
\Lambda,\,
A(IC^H_{X_0})\to A(\mathcal P^H)\to \bigoplus_{k=1}^r A(i_{k*}\mathbb Q^H_{\{p_k\}}(-1))
\bigr)
\]
should therefore be viewed as the finite-node mathematical state space underlying a coupled conifold light-state picture. It records:
\begin{enumerate}
\item the local rank-one sectors;
\item the global assembly law;
\item the transport interaction data;
\item the atom-side mixing behavior.
\end{enumerate}

In this sense, \(\mathfrak L_\Sigma\) is the mathematical precursor to a later multi-node reformulation of Strominger's conifold mechanism. It identifies the finite-node structure on which such a coupled light-state theory would have to be built, but it does not yet constitute the full physical recast itself.

\section{Examples}

We continue in the general setting of Section~\ref{subsec:general-setting}. In this section we first record a genuinely geometric multi-node conifold example, and then several explicit finite-node \emph{model configurations} of the package
\[
\mathfrak L_\Sigma
=
\Bigl(
\Sigma,\,
E_{\Sigma}^{\mathrm{node}},\,
E_{\Sigma}^{\mathrm{geom}},\,
\{T_i\}_{i=1}^r,\,
\Lambda,\,
A(IC^H_{X_0})\to A(\mathcal P^H)\to \bigoplus_{k=1}^r A(i_{k*}\mathbb Q^H_{\{p_k\}}(-1))
\Bigr).
\]
The point of the section is threefold:
\begin{enumerate}
\item to anchor the discussion with at least one genuine geometric multi-node conifold example;
\item to show concretely how split, interacting, and partially interacting behavior are reflected in the corrected-extension, transport, and atom-side realizations;
\item to provide model configurations to which later transport, recast, and BPS-oriented constructions may be applied.
\end{enumerate}

Unless explicitly stated otherwise, the examples after the first are presented as finite-node model configurations illustrating the package formalism. Their purpose is to display the behavior of the package in concrete low-dimensional patterns, not to assert here that every displayed pattern is already realized by an explicit projective Calabi--Yau degeneration.

Throughout, we use the nodewise decomposition
\[
E_{\Sigma}^{\mathrm{node}}
\cong
\bigoplus_{k=1}^r \mathbb Qe_k
\]
from \cite{RahmanQuiverI}, the relation-controlled gluing framework from \cite{RahmanGlobalGluing}, and the atom/Stokes realization from \cite{RahmanHodgeAtoms2026}.

\begin{remark}\label{rem:geometric-multinode-gap}
At present, this section contains one genuine geometric multi-node conifold example together with several finite-node model configurations. Producing a concrete projective Calabi--Yau degeneration for which the corrected-extension, transport, and atom-side realizations of finite-node interaction are all exhibited explicitly and in full remains an important open problem for the present program.
\end{remark}

\subsection{Geometric Example \#1: A simple one-node conifold baseline}

A genuine geometric example is the simple conifold degeneration considered by Strominger: a one-parameter Calabi--Yau degeneration with a single vanishing three-cycle and simple Picard--Lefschetz monodromy \cite{Strominger95}. In that setting one has
\[
\Sigma=\{p\},
\]
so there is only one local singular sector. The corrected-extension space is one-dimensional,
\[
E_{\Sigma}^{\mathrm{node}}\cong \mathbb Q e,
\]
and there is no nontrivial distinction between ambient and realized nodewise directions:
\[
E_{\Sigma}^{\mathrm{geom}}=E_{\Sigma}^{\mathrm{node}}.
\]
On the transport side, there is a single vanishing cycle \(\delta\), so there are no off-diagonal interaction terms. Equivalently, the transport-side interaction matrix is the \(1\times 1\) zero matrix,
\[
\Lambda=(0),
\]
and there is no multi-node commutator phenomenon. On the atom side, the corrected object carries a single flexible sector over the rigid bulk sector.

Thus the simple conifold gives a concrete geometric realization of the \emph{local building block} from which the later finite-node package is assembled: one node, one rank-one local sector, and no multi-node interaction. It is not yet an interacting finite-node example, but it is the honest geometric baseline from which the finite-node problem begins.

\begin{remark}
This example is included for geometric anchoring. The genuinely new content of the present paper begins when one passes from this one-node geometric baseline to finite-node interaction data.
\end{remark}

\subsection{Geometric Example \#2: Dwork/quintic conifold member with 125 nodes}

Another geometric example is the classical one-parameter Dwork/quintic conifold family \cite{COGP}
\[
X_\psi:\quad
x_0^5+x_1^5+x_2^5+x_3^5+x_4^5-5\psi\,x_0x_1x_2x_3x_4=0
\subset \mathbb P^4,
\]
viewed as a family of Calabi--Yau threefolds parametrized by \(\psi\). This family provides a genuine multi-node degeneration in a one-parameter Calabi--Yau family, exactly of the type that motivates the finite-node package of the present paper. The full finite-node package extraction for this family is deferred to future work. What follows is a concrete geometric test-case reading of the example in the language of the present paper.

For generic \(\psi\), the fiber \(X_\psi\) is smooth. At the conifold parameter values, equivalently when $\psi^5=1$, the family develops a singular member with finitely many ordinary double points. More precisely, the conifold fiber carries $|\Sigma|=125$ nodes, so this is a genuine projective Calabi--Yau degeneration with many simultaneously realized ODPs in a single one-parameter family.

This example is important because it already exhibits the key conceptual feature that motivates the present paper: the passage from the one-node conifold mechanism to the genuinely finite-node
regime does not require changing from a one-parameter family to a many-parameter family. A single degeneration parameter may already produce a singular fiber with many nodes at once. In
particular, the finite-node problem is not an artificial combinatorial enlargement of the one-node story; it occurs naturally in the classical conifold geometry underlying the quintic example. From the point of view of the present paper, the Dwork/quintic conifold member should be read as follows.

\paragraph{(1) Local node sectors.}
Each point \(p_k\in\Sigma\) is an ODP, so each contributes a rank-one local vanishing sector.
Thus the local finite-node input has the form
\[
\Sigma=\{p_1,\dots,p_{125}\}
\quad\rightsquigarrow\quad
\{\text{125 local rank-one sectors}\}.
\]
Accordingly, the ambient corrected-extension space is formally of the form
\[
E_{\Sigma}^{\mathrm{node}}
\cong
\bigoplus_{k=1}^{125}\mathbb Q e_k,
\]
at least at the level of nodewise state-data.

\paragraph{(2) Global assembly.}
Since the singular fiber comes from one projective degeneration rather than \(125\) unrelated
local models, the geometrically realized subspace
\[
E_{\Sigma}^{\mathrm{geom}}\subseteq E_{\Sigma}^{\mathrm{node}}
\]
is, in general, a smaller quotient cut out by global gluing constraints. Thus the first genuinely
global question in this family is not whether the \(125\) local sectors exist, but how many
independent global directions survive after imposing the geometric relation law.

\paragraph{(3) Transport-side meaning.}
Since the degeneration is one-parameter, the classical moduli-space description detects a conifold locus and its monodromy. But from the fiber-side point of view emphasized here, the presence of
many nodes raises a subtler question: how should the total monodromy be understood relative to the individual vanishing sectors? The package of the present paper records this through the transport-side data
\[
\{T_i\}_{i=1}^{125},
\qquad
\Lambda=(\lambda_{ij}),
\qquad
\lambda_{ij}=\langle \delta_i,\delta_j\rangle.
\]
At the level of the present example, we do not yet compute the full \(125\times125\) interaction
matrix \(\Lambda\), but the example shows why such a matrix is a natural object to seek: the total
monodromy of the conifold fiber alone does not transparently display the internal coupling
structure among the simultaneously vanishing sectors.

\paragraph{(4) Atom-side meaning.}
Likewise, on the atom side, the corrected object carries one flexible rank-one sector per node over
the rigid bulk sector. Thus the geometric example suggests an atom-side exact sequence of the form
\[
0\to A(IC^H_{X_0})\to A(\mathcal P^H)\to
\bigoplus_{k=1}^{125}A(i_{k*}\mathbb Q^H_{\{p_k\}}(-1))\to 0,
\]
whose splitting behavior would encode whether the \(125\) local sectors remain independent or mix nontrivially. The point of the present paper is not that this full splitting computation has
already been carried out for the quintic family, but that the package \(L_\Sigma\) identifies the precise finite-node structure one would need to compute.

\paragraph{(5) What the package would compute for this family.}
A fully worked application of the package \(L_\Sigma\) to the Dwork/quintic conifold member would require four pieces of data:
\begin{enumerate}
\item the actual relation-block decomposition
\[
\Sigma=\bigsqcup_{\beta\in B}\Sigma_\beta;
\]
\item the corresponding realized quotient
\[
E_{\Sigma}^{\mathrm{geom}}\cong \mathbb Q^{125}/R_{\mathrm{ext}};
\]
\item the associated blockwise vanishing classes
\[
v_\beta\in H^3(X_t;\mathbb Q);
\]
\item the reduced block interaction matrix
\[
\Lambda_{\mathrm{blk}}=(\mu_{\beta\gamma}),
\qquad
\mu_{\beta\gamma}:=\langle v_\beta,v_\gamma\rangle.
\]
\end{enumerate}
In other words, the package does not merely ask whether the quintic conifold member has many nodes; it asks how many independent global light-sector directions actually survive, and how the surviving directions interact.

\paragraph{(6) Symmetry-orbit heuristic.}
The Dwork/quintic family carries a large diagonal symmetry, and the \(125\) nodes of the conifold member occur in a highly symmetric configuration. It is therefore natural to expect that nodes
lying in the same symmetry orbit should contribute geometrically related local sectors, and hence should be candidates for belonging to the same relation block in the sense of Theorem~6.6. The
present paper does not verify the full block-separation hypothesis for this family, but this symmetry-orbit picture explains why the quintic conifold member is a natural testing ground for
the block-reduced structure theorem.

Conditionally, if the relation blocks coincide with the relevant symmetry orbits, then the finite-node quotient would be governed by the number of such symmetry-induced blocks rather than by the raw node count \(125\). In that case Theorem~6.6 would predict:
\begin{enumerate}
\item a realized quotient dimension
\[
\dim_{\mathbb Q}E_{\Sigma}^{\mathrm{geom}}=|B|
\]
strictly smaller than \(125\) whenever several nodes lie in a common block;
\item a reduced transport/atom interaction law governed by
\[
\Lambda_{\mathrm{blk}}=(\mu_{\beta\gamma})
\]
rather than by the full nodewise matrix \(\Lambda\).
\end{enumerate}

\begin{remark}
The finite-node package $\mathfrak L_\Sigma$ introduced here is exactly the kind of fiber-side structure needed to organize what the moduli-space picture alone does not transparently separate: the globally realized dimension, the nodewise assembly constraints, and the coupling and splitting structure among local vanishing sectors. We highlight the point that this example is a viable geometric test case for the finite-node light-sector package and its block-reduced structure theorem.
\end{remark}

\subsection{$A_1\times A_1$: split two-node model}

We now pass from geometric examples to controlled finite-node package models. The first is the split two-node configuration, which should be read as the minimal finite-node analogue of two independent one-node sectors.

Let
\[
\Sigma=\{p_1,p_2\}
\]
and assume
\[
\lambda_{12}=\lambda_{21}=0.
\]
Then
\[
\Lambda=
\begin{pmatrix}
0&0\\
0&0
\end{pmatrix}.
\]
Assume further that the cycle-node incidence law is trivial, so that the geometrically realized
subspace equals the full ambient nodewise space:
\[
E_{\Sigma}^{\mathrm{geom}}=E_{\Sigma}^{\mathrm{node}}
\cong
\mathbb Q e_1\oplus \mathbb Q e_2.
\]
Then the corrected class has the form
\[
[\mathcal P]_{\mathrm{perv}}=c_1e_1+c_2e_2
\]
with no relation between \(c_1\) and \(c_2\).

On the transport side,
\[
[T_1,T_2]=0,
\qquad
[N_1,N_2]=0.
\]
On the atom side, \cite[Thm.~1.4]{RahmanHodgeAtoms2026} gives splitting:
\[
0\to A(IC^H_{X_0})\to A(\mathcal P^H)\to
A(i_{1*}\mathbb Q^H_{\{p_1\}}(-1))\oplus
A(i_{2*}\mathbb Q^H_{\{p_2\}}(-1))
\to 0
\]
is split.

Thus this model exhibits the split pattern:
\[
E_{\Sigma}^{\mathrm{geom}}=E_{\Sigma}^{\mathrm{node}},
\qquad
\Lambda=0,
\qquad
A(\mathcal P^H)\cong
A(IC^H_{X_0})\oplus
A(i_{1*}\mathbb Q^H_{\{p_1\}}(-1))\oplus
A(i_{2*}\mathbb Q^H_{\{p_2\}}(-1)).
\]

\begin{proposition}\label{prop:A1xA1-example}
In the \(A_1\times A_1\) model, the finite-node light-sector package is split.
\end{proposition}

\begin{proof}
The corrected-extension realization is split because
\[
E_{\Sigma}^{\mathrm{geom}}=E_{\Sigma}^{\mathrm{node}}.
\]
The transport realization is split because
\[
\Lambda=0.
\]
The atom realization is split because the flexible-sector exact sequence splits. These are
exactly the split conditions recorded in Sections~3--6 for the corresponding realizations.
\end{proof}

This example is the correct control case for the genuinely finite-node regime. It shows that
the local rank-one sectors may remain globally independent, and therefore provides the baseline
against which interacting models should be measured.

\subsection{$A_2$: interacting two-node model}

The next model is the smallest genuinely interacting finite-node configuration. It should be
viewed as the minimal case in which the corrected-extension, transport, and atom-side
manifestations of interaction are all simultaneously visible.

Let
\[
\Sigma=\{p_1,p_2\}
\]
and assume
\[
\lambda_{12}\neq 0.
\]
By antisymmetry of the degree-three intersection pairing,
\[
\lambda_{21}=-\lambda_{12}\neq 0.
\]
Hence
\[
\Lambda=
\begin{pmatrix}
0&\lambda_{12}\\
-\lambda_{12}&0
\end{pmatrix}
\neq 0.
\]

Assume the global gluing law identifies the two nodewise directions into a one-dimensional
relation-controlled subspace:
\[
E_{\Sigma}^{\mathrm{geom}}
=
\mathbb Q(e_1+e_2)
\subsetneq
\mathbb Qe_1\oplus\mathbb Qe_2
=
E_{\Sigma}^{\mathrm{node}}.
\]
Then the corrected class has the form
\[
[\mathcal P]_{\mathrm{perv}}=c(e_1+e_2),
\qquad
c\in\mathbb Q.
\]

On the transport side,
\[
[N_1,N_2](\alpha)
=
\langle \alpha,\delta_2\rangle \lambda_{21}\delta_1
-
\langle \alpha,\delta_1\rangle \lambda_{12}\delta_2,
\]
hence
\[
[N_1,N_2]\neq 0,
\qquad
[T_1,T_2]\neq 0.
\]
On the atom side, the exact sequence
\[
0\to A(IC^H_{X_0})\to A(\mathcal P^H)\to
A(i_{1*}\mathbb Q^H_{\{p_1\}}(-1))\oplus
A(i_{2*}\mathbb Q^H_{\{p_2\}}(-1))
\to 0
\]
is non-split by \cite[Thm.~1.4]{RahmanHodgeAtoms2026}.

Thus this model exhibits interaction in all three packaged realizations:
\[
E_{\Sigma}^{\mathrm{geom}}\subsetneq E_{\Sigma}^{\mathrm{node}},
\qquad
\Lambda\neq 0,
\qquad
A(\mathcal P^H)\ \text{non-split}.
\]

\begin{proposition}\label{prop:A2-example}
In the \(A_2\) model, the finite-node light-sector package is interacting in each of the
three packaged realizations.
\end{proposition}

\begin{proof}
The first condition is the corrected-extension-side interacting criterion. The second is the
transport-side interacting criterion. The third is the atom-side interacting criterion. Each
follows from the assumptions of the model together with the results of Sections~3--6.
\end{proof}

This is the smallest nontrivial interacting example. Its importance is not merely pedagogical:
it shows that the naive expectation of two freely independent rank-one sectors can fail already
in the two-node case. The correct finite-node object is not the raw nodewise sum, but the
package \(\mathfrak L_\Sigma\) built from the realized quotient together with the transport and
atom-side interaction data.

\subsection{Three-node block-incidence model}

The next model exhibits the phenomenon that interaction need not be all-or-nothing. It is
the simplest package-level example in which one sees both relation collapse and partial
decoupling at the same time.

Let
\[
\Sigma=\{p_1,p_2,p_3\}
\]
and assume the cycle-node incidence law has two relation blocks
\[
\Sigma_1=\{p_1,p_2\},
\qquad
\Sigma_2=\{p_3\}.
\]
Equivalently, assume the admissible coefficient space is
\[
E_{\Sigma}^{\mathrm{geom}}
=
\operatorname{Span}_{\mathbb Q}\{e_1+e_2,\ e_3\}
\subsetneq
\operatorname{Span}_{\mathbb Q}\{e_1,e_2,e_3\}
=
E_{\Sigma}^{\mathrm{node}}.
\]
Then the corrected class has the form
\[
[\mathcal P]_{\mathrm{perv}}=a(e_1+e_2)+be_3,
\qquad
a,b\in\mathbb Q.
\]

Assume moreover that
\[
\lambda_{12}\neq 0,
\qquad
\lambda_{13}=\lambda_{23}=0.
\]
By antisymmetry,
\[
\lambda_{21}=-\lambda_{12}\neq 0.
\]
Then
\[
\Lambda=
\begin{pmatrix}
0&\lambda_{12}&0\\
-\lambda_{12}&0&0\\
0&0&0
\end{pmatrix}.
\]
Thus \((p_1,p_2)\) form an interacting block, while \(p_3\) is split from that block.

From the corrected-extension point of view, the first two nodes contribute only one global
direction \(e_1+e_2\), while the third survives independently. From the transport point of
view, coupling occurs only within the first block. On the atom side, the transport/atom
comparison suggests that
\[
A(i_{1*}\mathbb Q^H_{\{p_1\}}(-1))
\quad\text{and}\quad
A(i_{2*}\mathbb Q^H_{\{p_2\}}(-1))
\]
mix nontrivially, while
\[
A(i_{3*}\mathbb Q^H_{\{p_3\}}(-1))
\]
remains independent.

\begin{proposition}\label{prop:three-node-block-example}
In the three-node block-incidence model,
\[
E_{\Sigma}^{\mathrm{geom}}
=
\operatorname{Span}_{\mathbb Q}\{e_1+e_2,\ e_3\},
\qquad
\Lambda=
\begin{pmatrix}
0&\lambda_{12}&0\\
-\lambda_{12}&0&0\\
0&0&0
\end{pmatrix},
\]
so the finite-node light-sector package decomposes into one interacting two-node block
together with one split one-node block.
\end{proposition}

\begin{proof}
Immediate from the definitions of \(E_{\Sigma}^{\mathrm{geom}}\) and \(\Lambda\).
\end{proof}

This example is important because it illustrates the same two-layer structure emphasized in
Theorem~6.6:
\[
\text{relation collapse}
\quad+\quad
\text{residual interaction among surviving sectors}.
\]
The corrected-extension side reduces the three raw nodewise directions to two surviving global
directions, while the transport-side matrix shows that only one of those two directions is
actually coupled.

\subsection{A comparative model across realizations}

We record the \(A_2\) model side by side in the three realizations packaged by
\(\mathfrak L_\Sigma\). The point of this subsection is not to add a new theorem, but to
show explicitly how the same finite-node configuration is read on the corrected-extension,
transport, and atom sides.

\paragraph{Corrected-extension side.}
The ambient nodewise space is
\[
E_{\Sigma}^{\mathrm{node}}=\mathbb Qe_1\oplus \mathbb Qe_2,
\]
while the geometrically realized subspace is
\[
E_{\Sigma}^{\mathrm{geom}}=\mathbb Q(e_1+e_2).
\]
Thus the two raw nodewise directions do not survive independently: the corrected-extension
package retains only one globally realized direction. This is the corrected-extension-side
manifestation of finite-node interaction.

\paragraph{Transport side.}
The interaction matrix is
\[
\Lambda=
\begin{pmatrix}
0&\lambda_{12}\\
-\lambda_{12}&0
\end{pmatrix}
\neq 0,
\]
and the Picard--Lefschetz operators fail to commute:
\[
[T_1,T_2]\neq 0.
\]
Equivalently, the nilpotent operators satisfy
\[
[N_1,N_2](\alpha)
=
\langle \alpha,\delta_2\rangle \lambda_{21}\delta_1
-
\langle \alpha,\delta_1\rangle \lambda_{12}\delta_2
\neq 0.
\]
Thus the transport-side manifestation of the same configuration is nontrivial coupling of
the two local vanishing sectors.

\paragraph{Atom side.}
The atom-side exact sequence is
\[
0\to A(IC^H_{X_0})\to A(\mathcal P^H)\to
A(i_{1*}\mathbb Q^H_{\{p_1\}}(-1))\oplus
A(i_{2*}\mathbb Q^H_{\{p_2\}}(-1))
\to 0,
\]
and it is non-split. Thus the two flexible rank-one sectors do not remain independent over
the rigid bulk sector; rather, they mix nontrivially.

Putting these three descriptions together, the \(A_2\) model displays the same finite-node
configuration in three compatible realizations:
\[
E_{\Sigma}^{\mathrm{geom}}\subsetneq E_{\Sigma}^{\mathrm{node}},
\qquad
\Lambda\neq 0,
\qquad
A(\mathcal P^H)\ \text{non-split}.
\]
The corrected-extension side records the collapse from two raw nodewise directions to one
realized global direction; the transport side records nontrivial coupling of the surviving
sectoral data; and the atom side records that this coupling appears as non-splitting of the
corrected atom package.

\begin{remark}
The displayed parallelism should be read in the sense of the compatibility framework of
Section~6. The present example is not intended as a new stand-alone proof that any one of
the three displayed conditions implies the other two in complete generality. Rather, it shows
how the package \(\mathfrak L_\Sigma\) allows the same finite-node interaction pattern to be
read concretely across the corrected-extension, transport, and atom realizations.
\end{remark}

\section{Outlook}

The constructions of the present paper suggest the following directions.

\begin{itemize}
\item \textbf{From the fiber-side package to moduli-space effective theory.}
Translate the finite-node package \(\mathfrak L_\Sigma\) into the moduli-space language of periods, monodromy, and effective couplings, with the goal of formulating a genuine multi-node analogue of Strominger's conifold mechanism.

\item \textbf{The multi-node conifold light-state mechanism.}
Reinterpret the package \(\mathfrak L_\Sigma\) as the finite-node state space underlying coupled conifold light states.

\item \textbf{Quotient and symmetry refinements.}
Study equivariant, quotient, and symmetry-adapted versions of the interacting light-sector package.

\item \textbf{BPS sheaves, halo sectors, and higher extensions.}
Relate the finite-node package to BPS sheaf constructions, halo-type sector formation, and higher-extension data.

\item \textbf{Wall-crossing on the interacting package.}
Formulate transport and wall-crossing structures acting directly on the interacting multi-node light-sector package.

\item \textbf{Geometric realization of interacting finite-node models.}
Construct explicit projective Calabi--Yau degenerations realizing interacting multi-node patterns, beyond the one-node geometric baseline recorded in Section~8.
\end{itemize}

%
%
\printbibliography

\end{document}